\documentclass{amsart}
\usepackage{amsmath,amssymb,amsthm,mathtools}
\usepackage{booktabs}

\usepackage{tikz-cd}
\tikzcdset{arrow style=tikz, diagrams={>=stealth}}

\usepackage{hyperref}

\usepackage{nicematrix}
\usepackage[boxruled,vlined,noend,noalgohanging,nofillcomment]{algorithm2e}

\usepackage{color}

\usepackage{changes}

\DeclareTextFontCommand{\texttt}{\ttfamily\upshape}

\theoremstyle{plain}
\newtheorem{theorem}{Theorem}[section]
\newtheorem{lemma}[theorem]{Lemma}
\newtheorem{proposition}[theorem]{Proposition}
\newtheorem{corollary}[theorem]{Corollary}
\numberwithin{equation}{section}

\theoremstyle{definition}

\newtheorem{example}[theorem]{Example}

\theoremstyle{remark}
\newtheorem{remark}[theorem]{Remark}

\newcommand{\Z}{\mathbb Z}
\newcommand{\R}{\mathbb R}
\newcommand{\F}{\mathbb F}

\newcommand{\cO}{\mathcal O}
\newcommand{\cI}{\mathcal I}
\newcommand{\cF}{\mathcal F}
\newcommand{\bI}{\mathbf I}
\newcommand{\scol}{\mathsf{col}}
\newcommand{\srow}{\mathsf{row}}

\DeclareMathOperator{\cone}{cone}
\DeclareMathOperator{\Wed}{Wed}
\DeclareMathOperator{\join}{\ast}

\DeclareMathOperator{\Proj}{Proj}
\DeclareMathOperator{\supp}{supp}
\DeclareMathOperator{\head}{head}
\DeclareMathOperator{\rank}{rank}
\DeclareMathOperator{\corank}{corank}

\title[Toric-Colorable Seeds of Picard Number Five]{Enumerating Toric-Colorable Seeds of Picard Number Five via Binary Matroids}

\date{\today}

\author{Suyoung Choi}
\address{Department of Mathematics, Ajou University, Suwon 16499, Republic of Korea}
\email{schoi@ajou.ac.kr}

\author{Mathieu Vall\'ee}
\address{Universit\'e libre de Bruxelles, CP212, Boulevard du Triomphe, 1050 Brussels, Belgium}
\email{mathieu.vallee@protonmail.com}


\thanks{Computational resources have been provided by the Consortium des Équipements de Calcul Intensif (CÉCI), funded by the Fonds de la Recherche Scientifique de Belgique (F.R.S.-FNRS) under Grant No. 2.5020.11 and by the Walloon Region.
This work was also supported by the National Research Foundation of Korea Grant funded by the Korean Government (RS-2025-00521982)
}

\begin{document}

\begin{abstract}
    We introduce a binary matroid framework for the enumeration of mod~$2$ toric-colorable seeds of fixed Picard number. Working with binary matroids up to isomorphism, we organize them through their contraction structure and recursively enumerate weak pseudomanifold subcomplexes by a dynamic programming algorithm. The resulting method combines these structural reductions with a Gray-code traversal of the mod~$2$ kernel of the ridge--facet incidence matrix.

    Using this framework, we complete the classification of mod~$2$ toric-colorable seeds of dimension four and Picard number five, proving that there are exactly $198{,}846$ isomorphism classes. This case was computationally infeasible for the GPU-based algorithm previously used by Choi, Jang, and Vall\'ee to treat Picard number four. We further verify that each of these seeds admits an integral characteristic map. As a validation of the method, we also reproduce that Picard number four classification: the weak pseudomanifold enumeration stage drops from over ten days on a GPU to ten minutes on a single CPU.
\end{abstract}

\maketitle

\section{Introduction}
A \emph{toric manifold} is a complete nonsingular toric variety.
By the fundamental theorem of toric geometry, toric manifolds are described by complete nonsingular fans.
This description turns geometric questions about toric manifolds into combinatorial questions about cones and lattice vectors.
Thus the classification of toric manifolds is closely related to the classification of certain simplicial spheres with characteristic maps.

Let $\Sigma$ be a complete nonsingular fan of dimension $n$, and let $\rho_1,\ldots,\rho_m$ be its rays.
For each $i$, let $v_i$ be the primitive generator of $\rho_i$.
The fan $\Sigma$ determines a simplicial complex $K_\Sigma=\{I\subset [m]\mid \cone(v_i\mid i\in I)\in \Sigma\}$.
Since $\Sigma$ is complete and nonsingular, $K_\Sigma$ is an $(n-1)$-dimensional piecewise linear (PL)~sphere on $[m]$.
The assignment of primitive ray generators gives a map $\lambda_\Sigma\colon [m]\to \Z^n$ given by $i\mapsto v_i$.
This map satisfies the characteristic condition
\begin{quote}
    $\{\lambda_\Sigma(i)\mid i\in I\}$ is part of a $\Z$-basis of $\Z^n$ for every $I\in K_\Sigma$.  
\end{quote}
Equivalently, the vectors assigned to every facet of $K_\Sigma$ form a $\Z$-basis of $\Z^n$.

A PL~sphere $K$ is called \emph{toric-colorable} if it supports a $\Z$-characteristic map.
It is called \emph{fanlike} if it supports a $\Z$-characteristic map whose associated cones form a complete nonsingular fan.
Thus every fanlike PL~sphere is toric-colorable, but the converse might not hold.
The fanlike condition is the condition needed to construct toric manifolds.
The toric-colorable condition is a necessary first condition.

In this paper, we mainly use the mod~$2$ version of this condition.
Let $K$ be an $(n-1)$-dimensional PL~sphere on $[m]$.
We say that $K$ is \emph{mod~$2$ toric-colorable}, or \emph{$\Z_2^n$-colorable}, if there exists a map $\lambda^\R\colon [m]\to \Z_2^n$ such that
\begin{quote}
    $\{\lambda^\R(i)\mid i\in I\}$ is linearly independent over $\Z_2$ for every $I\in K$.  
\end{quote}
Equivalently, the vectors assigned to every facet of $K$ form a basis of $\Z_2^n$.
Every $\Z$-characteristic map reduces to such a mod~$2$ characteristic map.
Hence every fanlike PL~sphere is mod~$2$ toric-colorable.
This makes mod~$2$ toric-colorability an important and computable necessary condition in the classification problem.

The mod~$2$ condition is rephrased as follows in terms of binary matroids.
Indeed, a map $\lambda^\R\colon [m]\to \Z_2^n$ defines a binary matroid $M_{\lambda^\R}$ on $[m]$, whose independent sets are precisely  those subsets $I\subset [m]$ for which $\{\lambda^\R(i)\mid i \in I\}$ is linearly independent.
Thus, after identifying a matroid with its independence complex, the mod~$2$ non-singularity condition for $\lambda^\R$ with respect to $K$ is equivalent to the inclusion
$$
    K\subseteq M_{\lambda^\R}.
$$
This reformulation is the main combinatorial idea used within this paper.

Throughout, $K$ is an $(n-1)$-dimensional PL~sphere on $[m]$.
Following the toric convention, we call $p=m-n$ the \emph{Picard number} of $K$.

The classification of toric manifolds by Picard number is a classical problem in algebraic geometry.
Toric manifolds of Picard number one are isomorphic to projective spaces.
Those of Picard number two were classified by Kleinschmidt~\cite{Kleinschmidt1988}.
The Picard number three case was studied by Batyrev~\cite{Batyrev1991}.
Recently, Choi, Jang, and Vall\'ee completed the classification of complete nonsingular fans of Picard number four in~\cite{Choi-Jang-Vallee2025}.

Every PL~sphere lies in the wedge class of a wedge-minimal PL~sphere, which is called a \emph{seed}.
Equivalently, every PL~sphere can be obtained from a seed by a sequence of simplicial wedge operations.
If $K$ has dimension $n-1$ and $m$ vertices, then a wedge operation increases both $n$ and $m$ by one.
Hence it preserves the Picard number.
By Choi and Park~\cite{Choi-Park2016Wedge}, every complete nonsingular fan over $K$ can be obtained from a complete nonsingular fan over a seed of $K$ by the corresponding sequence of wedge operations.
Therefore, the classification of toric manifolds with fixed Picard number reduces to the classification of fanlike seeds with that Picard number and the classification of characteristic maps over them that give complete nonsingular fans.

The finiteness theorem in~\cite{Choi-Park2017WedgeOperationsTorus} states that, for each fixed Picard number, there are only finitely many mod~$2$ toric-colorable seeds.
Thus, the enumeration of mod~$2$ toric-colorable seeds, and hence the search for fanlike seeds among them, is a finite problem.

Choi, Jang, and Vall\'ee carried out this program for Picard number four.
First, they enumerated all mod~$2$ toric-colorable seeds of Picard number four by using a GPU algorithm, and found $3{,}153$ such seeds~\cite{Choi-Jang-Vallee2024}.
Second, they determined which of these seeds are fanlike by testing which characteristic maps give complete nonsingular fans, and found $59$ fanlike seeds.
In addition, they classified the characteristic maps over these seeds that yield complete nonsingular fans.
This provides the classification of toric manifolds with Picard number four~\cite{Choi-Jang-Vallee2025}.

In this paper, we introduce a binary matroid approach to the enumeration of mod~$2$ toric-colorable seeds.

By Corollary~\ref{cor:non_IDCM_colorable_seeds}, every mod~$2$ toric-colorable seed which is not the boundary of a cross-polytope is a subcomplex of a cosimple binary matroid of the same Picard number; the excluded cross-polytope boundaries are reintroduced directly in Section~\ref{sec:postprocessing}. Consequently, the enumeration of mod~$2$ toric-colorable seeds via cosimple binary matroids reduces to the finite collection of cosimple binary matroids of fixed corank, which for such matroids coincides with the Picard number~$p$.
Since taking the link of a vertex corresponds to contracting the associated binary matroid, we organize these matroids according to their contraction structure. This leads to a dynamic programming algorithm that recursively enumerates weak pseudomanifold subcomplexes, while a Gray-code traversal accelerates the computation over the mod~$2$ kernel of the ridge--facet incidence matrix. The resulting complexes are reduced by the automorphism group of their ambient cosimple binary matroid, then filtered by the PL sphere and seed conditions, and finally reduced up to isomorphism.

In particular, we complete here the new case $(n,p)=(5,5)$, that is, dimension four and Picard number five.
We also verify, in Section~\ref{sec:toric-colorable}, that every resulting mod~$2$ toric-colorable seed PL~sphere admits an integral characteristic map; equivalently, our enumeration is also the enumeration of toric-colorable seeds of dimension four and Picard number five.
This case cannot be obtained from the existing database of PL~spheres, since it would require a complete database of PL~$4$-spheres with ten vertices.
Moreover, the GPU-accelerated algorithm of~\cite{Choi-Jang-Vallee2024} was unable to terminate on this case of PL~$4$-spheres with ten vertices.
Our algorithm gives the following value.

\begin{theorem}\label{thm-main}
    The number of toric-colorable seeds of dimension four and Picard number five is $198{,}846$ up to combinatorial isomorphism.
\end{theorem}

It is worth noting that the same framework also reproduces the Picard number four enumeration of~\cite{Choi-Jang-Vallee2024}.
This provides an independent verification of the data in~\cite{Choi-Jang-Vallee2024} by a different method.
The known values are summarized in Table~\ref{table:result}.

\begin{table}[ht]
    \caption{Known numbers of toric-colorable seeds of dimension $n-1$ and Picard number $p=m-n$.}
    \label{table:result}
    \centering
    \small
    \setlength{\tabcolsep}{5pt}
    \begin{tabular}{c c c c c c c}
    \toprule
    $n \backslash p$
    & $1$ & $2$ & $3$ & $4$ & $5$ & $6$ \\
    \midrule
    $1$  & $1$ &     &     &       &                  &                  \\
    $2$  &     & $1$ & $1$ & $1$   & $1$              & $1$              \\
    $3$  &     &     & $1$ & $4$   & $13$             & $49$             \\
    $4$  &     &     & $1$ & $21$  & $1{,}170^\dagger$ & $213{,}841^\dagger$ \\
    $5$  &     &     &     & $142$ & $\mathbf{198{,}846}$ & $*$          \\
    $6$  &     &     &     & $733$ & $*$              & $*$              \\
    $7$  &     &     &     & $1{,}190$ & $*$          & $*$              \\
    $8$  &     &     &     & $776$ & $*$              & $*$              \\
    $9$  &     &     &     & $243$ & $*$              & $*$              \\
    $10$ &     &     &     & $39$  & $*$              & $*$              \\
    $11$ &     &     &     & $4$   & $*$              & $*$              \\
    $12$--$26$ & &   &     &       & $*$              & $*$              \\
    $>27$      & &   &     &       &                  & $*$              \\
    \bottomrule
    \end{tabular}
\end{table}

Here a blank entry means that no such seed occurs, while an asterisk indicates that the corresponding enumeration has not yet been carried out.
The bold entry is the new main result of this paper.
The entries marked by the symbol~$^\dagger$ are the three-dimensional values first recorded here using Lutz's complete database of PL~$3$-spheres\footnote{\href{https://www3.math.tu-berlin.de/IfM/Nachrufe/Frank_Lutz/stellar/}{The Manifold Page}} and the Garrison--Scott colorability test~\cite{Garrison-Scott2003}.
The value $1{,}170^\dagger$ is also independently reproduced by our dynamic programming algorithm.
Finally, note that by the four-color theorem, every PL $2$-sphere is mod~$2$ toric-colorable. 
Among them, the wedge of an $(m-1)$-gon is the unique non-seed.
Therefore, the entries in the row $n=3$ are given by $a_{m}-1$, where $a_m$ denotes the number of PL~$2$-spheres with $m$ vertices, recorded as OEIS~A000109\footnote{\url{https://oeis.org/A000109}}.

\subsection*{Code and dataset availability}

    The code is available at \url{https://github.com/MVallee1998/Pic5_public}.
    The complete lists of toric-colorable seeds of Picard number~$4$ 
    ($3{,}153$ seeds), Picard number~$5$ accross all dimensions up to four ($200{,}030$ seeds in total, including the $198{,}846$ of dimension four established in Theorem~\ref{thm-main}), and Picard number~$6$ accross all dimensions up to three ($213{,}891$ seeds in total, obtained from the database of F. Lutz) are deposited at \href{https://doi.org/10.5281/zenodo.21106287}{doi:10.5281/zenodo.21106287}.

\section{Preliminaries}\label{sec:prelim}

\subsection{Simplicial complexes, PL~spheres, and wedge operation}

In all that follows, we let $[m]:=\{1,2,\ldots,m\}$ for a positive integer $m$.
A \emph{simplicial complex}~$K$ on the ambient vertex set $[m]$ is a collection of subsets of $[m]$ such that if $\sigma\in K$ and $\tau\subset \sigma$, then $\tau\in K$.
An element of $K$ is called a \emph{face} of $K$, and a face with $k+1$ elements is called a \emph{$k$-face}.
An element $v\in [m]$ is called a \emph{ghost} vertex of $K$ if $\{v\}\notin K$.
Let $K$ be a simplicial complex on $V$ and let $K'$ be a simplicial complex on $V'$.
We say that $K$ and $K'$ are \emph{isomorphic}, written $K\cong K'$, if there exists a bijection $\phi\colon V\to V'$ such that $K'=\{\phi(\sigma)\mid \sigma\in K\}$.

A simplicial complex $K$ is called a \emph{PL~sphere} if its geometric realization is piecewise-linearly homeomorphic to the boundary of a simplex, see~\cite{Rourke-Sanderson1972book} for more details.
Every PL~sphere is \emph{pure}.
This means that all its maximal faces, with respect to inclusion, have the same dimension.
If this dimension is $n-1$, then the maximal faces are called the \emph{facets}, and the faces of dimension~$n-2$ are called the \emph{ridges}.
The \emph{dimension} of a pure simplicial complex $K$ is the dimension of its maximal faces.
The \emph{Picard number} of a simplicial complex $K$ on $[m]$ with $k$ ghost vertices, and of dimension $n-1$ is defined as $m-n-k$.
In particular, if $K$ has no ghost vertices, then its Picard number is $m-n$.

Unless otherwise stated, PL~spheres considered in enumeration are assumed to have no ghost vertices. 
Ghost vertices are allowed mainly when we regard inclusions $K\subseteq M$ and links or contractions as complexes on an ambient vertex set.

The \emph{link} of a face $\sigma$ in $K$ is the subcomplex
$$
    K/\sigma :=\{\tau\in K \mid \sigma\cap \tau = \emptyset, \sigma \cup \tau \in K\}.
$$
We regard $K/\sigma$ as a simplicial complex on the ambient vertex set $[m]\setminus \sigma$.
In particular, it may have ghost vertices.
When $\sigma = \{v\}$, we denote the link by $K/v$.

The \emph{join} of two simplicial complexes $K_1$ and $K_2$ on disjoint vertex sets is the simplicial complex
$$
    K_1 \join K_2 := \{\sigma_1 \cup \sigma_2 \mid \sigma_1 \in K_1, \sigma_2 \in K_2\}.
$$
In all that follows, we let $I$ be a $1$-simplex with vertices $v_1$ and $v_2$, and $\partial I$ be its boundary, that is, the $0$-sphere.
The \emph{suspension} of a simplicial complex $K$ is the join of $K$ with a $0$-sphere.
The vertices $v_1$ and $v_2$ are called \emph{suspended vertices}.

The \emph{(simplicial) wedge operation} on a simplicial complex $K$ at a non-ghost vertex~$v$ of $K$ is the simplicial complex
$$
    \Wed_v(K) := (I \join K/v) \cup (\partial I \join (K \setminus v)),
$$
where $K \setminus v$ is the subcomplex of $K$ consisting of faces not containing $v$.
In that case, the vertices $v_1$ and $v_2$ are called \emph{wedged vertices}.
Intuitively, the wedge operation replaces the vertex $v$ with an edge $I$.
Faces in the link of $v$ are joined with this edge, while faces not containing $v$ are joined with the boundary $\partial I$.
The wedge operation increases both the dimension and the number of vertices of $K$ by one, hence preserving the Picard number.
It also preserves the PL~sphere property~\cite{Choi-Park2016Wedge}.
By iterating the wedge operation, we can obtain new simplicial complexes from a given one.
A PL~sphere $K$ is called a \emph{seed} if it is wedge-minimal, that is, if it cannot be obtained from a PL~sphere of smaller dimension by a sequence of wedge operations.
Hence, every PL~sphere is obtained through iterated wedges starting from a seed.
Note that wedging at either one of the two wedged vertices coming from a wedge at $v$ yields isomorphic simplicial complexes.

\subsection{Toric-colorable PL~spheres and characteristic maps}
A \emph{characteristic map} over an $(n-1)$-dimensional simplicial complex~$K$ on the vertex set $[m]$ is a map
$$
    \lambda\colon [m] \to \Z^n
$$
such that for each face $\sigma = \{i_1, \ldots, i_k\} \in K$, the set of vectors $\{\lambda(i_1), \ldots, \lambda(i_k)\}$ can be extended to a basis of $\Z^n$, a condition referred to as the \emph{nonsingularity condition}.
A pair $(K,\lambda)$ is called a \emph{characteristic pair}.
A simplicial complex $K$ is called \emph{toric-colorable} if there exists a characteristic map over $K$.

A \emph{mod~$2$ characteristic map} is a map $\lambda^\R\colon [m] \to \Z_2^n$
\footnote{The superscript~$\R$ follows the convention in toric topology, where such mod~$2$ characteristic maps arise in the study of real toric manifolds.} 
such that for each face $\sigma = \{i_1, \ldots, i_k\} \in K$, the set of vectors $\{\lambda^\R(i_1), \ldots, \lambda^\R(i_k)\}$ is linearly independent over $\Z_2$, a condition referred to as the \emph{mod~$2$ nonsingularity condition}.
If $K$ supports a mod~$2$ characteristic map, then we say that $K$ is \emph{mod~$2$ toric-colorable}.
In particular, if $K$ supports a characteristic map $\lambda$, then its mod~$2$ reduction is a mod~$2$ characteristic map over $K$, hence toric colorability implies mod~$2$ toric-colorability.

Let $W$ be part of a basis of $\Z^n$.
The \emph{projection} of $\Z^n$ with respect to $W$ is the quotient map
$$
    \pi_{/W}\colon \Z^n \to \Z^n / \langle w_1,\ldots,w_k \rangle
$$
followed by a chosen identification of the quotient with $\Z^{n-k}$.
Let $\lambda\colon [m] \to \Z^n$ be a map.
A subset $\sigma\subseteq [m]$ is called \emph{nonsingular with respect to $\lambda$} if the set of vectors $\{\lambda(i) \mid i \in \sigma\}$ is part of a basis of $\Z^n$.
Given a nonsingular subset $\sigma\subseteq [m]$, the projection of $\lambda$ with respect to $\sigma$ is the composite map
$$
    \begin{array}{rcl}
        \Proj_\sigma \lambda \colon& [m]\setminus \sigma &\to \Z^{n-|\sigma|}\\
            & i &\mapsto \pi_{/\{\lambda(v) \mid v \in \sigma\}}(\lambda(i)).
    \end{array}
$$
In particular, if $\lambda$ is a characteristic map over $K$ and $\sigma$ is a face of $K$, then $\Proj_\sigma \lambda$ is a characteristic map over $K /\sigma$.

The same definition applies over $\Z_2$.
If $\lambda^\R\colon [m]\to\Z_2^n$ and $\sigma$ is nonsingular, then $\Proj_\sigma\lambda^\R$ is obtained by quotienting $\Z_2^n$ by the span of $\{\lambda^\R(i)\mid i\in\sigma\}$.

\subsection{A binary matroid rephrasing}

A \emph{matroid} is a pure simplicial complex which satisfies the following \emph{augmentation property}: if $\tau$ and $\sigma$ are two faces of the matroid such that $|\tau| < |\sigma|$, then there exists an element $x \in \sigma \setminus \tau$ such that $\tau \cup \{x\}$ is also a face of the matroid.
The underlying set of the matroid is called \emph{ground set}.
The faces of the matroid are called \emph{independent sets}, the maximal faces are called \emph{bases}, and their common size is called the \emph{rank} of the matroid.
Note that the rank of a matroid equals its dimension plus one.
An \emph{$\F$-representable matroid} is a matroid that can be represented as a collection of linearly independent subsets of columns of a matrix over the field $\F$.
A matroid is \emph{binary} if it is representable over the finite field $\Z_2$.

In particular, any given mod~$2$ map $\lambda^\R\colon [m]\to\Z_2^n$ encodes a binary matroid~$M_{\lambda^\R}$ on the ground set~$[m]$.
Its rank is at most $n$.
Moreover, the mod~$2$ nonsingularity condition with respect to an $(n-1)$-dimensional PL~sphere $K$ is reinterpreted as an inclusion of simplicial complexes
$$K \subseteq M_{\lambda^\R}.$$
Therefore, $K$ is mod~$2$ toric-colorable if and only if there exists a binary matroid~$M$ of rank $n$ on the ground set $[m]$ such that $K \subseteq M$.

The projection of mod~$2$ characteristic maps can also be rephrased in terms of binary matroids, as follows.
\begin{proposition}\label{prop:projection_matroid}
    Let $\lambda^\R\colon [m]\to \Z_2^n$ be a mod~$2$ map, and let $M_{\lambda^\R}$ be its associated binary matroid.
    Suppose that $\sigma\subset [m]$ is nonsingular with respect to $\lambda^\R$.
    After identifying both ground sets with $[m]\setminus\sigma$, we have
    $$
        M_{\Proj_\sigma \lambda^\R}=M_{\lambda^\R}/\sigma.
    $$
\end{proposition}
\begin{proof}
Let $\tau\subset [m]\setminus\sigma$.
By the definition of the projected map, the vectors $\{(\Proj_\sigma\lambda^\R)(i)\mid i\in\tau\}$ are linearly independent if and only if $\{\lambda^\R(i)\mid i\in \tau\cup\sigma\}$ are linearly independent.
Thus $\tau$ is independent in $M_{\Proj_\sigma\lambda^\R}$ if and only if $\tau\cup\sigma$ is independent in $M_{\lambda^\R}$.
This is exactly the definition of the \emph{contraction} $M_{\lambda^\R}/\sigma$ in matroid theory, see~\cite[Proposition~3.2.4]{Oxley2011book}.
\end{proof}

Hence, suppose that $K\subseteq M$, where $K$ is a PL~sphere and $M$ is a binary matroid.
For every face $\sigma$ of $K$, we obtain the following diagram.
$$
    \begin{tikzcd}
        K \arrow[r, hook] \arrow[d,hookleftarrow] & M \arrow[d, hookleftarrow] \\
        K /\sigma \arrow[r, hook] & M/ \sigma
    \end{tikzcd}
$$
Here the vertical arrows indicate the natural inclusions of links and contractions into the corresponding independence complexes.

These properties allow us to work with binary matroids instead of mod~$2$ characteristic maps.
In this language, projections of characteristic maps become contractions of binary matroids.
Equivalently, at the level of independence complexes, they become links of independent sets.

A \emph{circuit} is a minimal dependent set.
Equivalently, it is a minimal non-face of the matroid.
A \emph{loop} in a matroid is a circuit of size one, that is, an element that is not part of any basis.
In the context of binary matroids, a loop corresponds to a column of the representing matrix that is all zero.
Two elements in a circuit of size two are called \emph{parallel}, and for binary matroids, they correspond to two identical nonzero columns in the representing matrix.
A matroid is called \emph{simple} if it has no loops or parallel elements.

The \emph{dual} of a matroid~$M$ on the ground set~$E$ is the matroid~$M^\ast$ on the same ground set $E$ whose bases are the complements of the bases of $M$.
A \emph{coloop} is a loop of the dual matroid $M^\ast$, or equivalently, an element that is contained in every basis of $M$.
A cocircuit is a circuit of the dual matroid $M^\ast$.
Two elements in a cocircuit of size two are called \emph{series elements}.
A matroid is called \emph{cosimple} if it has no coloops or series elements.

We have the following elementary facts, see \cite[3.1.1, 3.1.5]{Oxley2011book}.
\begin{proposition}\label{proposition:contraction_matroids}
    Let $M$ be a cosimple binary matroid on the ground set $[m]$ with corank~$p$.
    If $v\in[m]$ is not a loop of $M$, then $M/v$ is a cosimple binary matroid on $m-1$ elements with corank $p$.
    More generally, if $\sigma$ is an independent set of $M$, then $M/\sigma$ is a cosimple binary matroid on $m-|\sigma|$ elements with corank $p$.
\end{proposition}
\begin{proof}
    Since $v$ is not a loop, we have $\rank(M/v)=\rank(M)-1$.
    Hence $\corank(M/v)=(m-1)-(\rank(M)-1)=\corank(M) = p$.
    Moreover, $(M/v)^\ast=M^\ast\setminus v$.
    Since $M$ is cosimple, $M^\ast$ is simple.
    Deletion preserves simplicity.
    Thus $(M/v)^\ast$ is simple, and so $M/v$ is cosimple.
    The general case is done by induction on $|\sigma|$.
\end{proof}

A convenient way to represent a binary matroid and its dual is to use the following standard form of their representing matrices which in turn correspond to mod~$2$ maps.
\begin{proposition}\label{prop:standard_form_matroid}
    Let $M$ be a binary matroid of rank $n$ on the ground set $[m]$.
    There exists a basis $B \subseteq [m]$ of $M$ ($|B|=n$) such that, after reordering the columns and applying elementary row operations, the representing matrix of $M$ is of the form
    $$
        \begin{bNiceArray}{c|c}[first-row]
             B & [m]\setminus B \\
            \bI_n & A
        \end{bNiceArray},
    $$
    where $\bI_n$ is the $n \times n$ identity matrix, and $A$ is an $n \times (m-n)$ matrix over $\Z_2$.
    
    Moreover, a representing matrix of the dual matroid $M^\ast$ is given by
    $$
        \begin{bNiceArray}{c|c}[first-row]
            B & [m]\setminus B \\
            A^\top & \bI_{m-n}
        \end{bNiceArray},
    $$
    and is called a \emph{Gale dual} matrix.
\end{proposition}

\begin{example}
    Let us consider the following matrix in $\Z_2$, and its Gale dual:
    $$
        \begin{bNiceArray}{ccc|cc}[first-row]
            1 & 2 & 3 & 4 & 5 \\
            1 & 0 & 0 & 0 & 0\\
            0 & 1 & 0 & 1 & 1\\
            0 & 0 & 1 & 1 & 1
        \end{bNiceArray}\quad \text{ and } \quad 
        \begin{bNiceArray}{ccc|cc}[first-row]
            1 & 2 & 3 & 4 & 5 \\
            0 & 1 & 1 & 1 & 0\\
            0 & 1 & 1 & 0 & 1
        \end{bNiceArray}.
    $$
    The associated binary matroid $M$ has ground set $\{1,2,3,4,5\}$ and its bases are 
    $$
        \{123,124,125,134,135\}.
    $$
    Here, for instance, $123$ denotes the subset $\{1,2,3\}$.
    The circuits of $M$ are $\{234,235,45\}$, hence the elements $4$ and $5$ are parallel.

    The dual matroid $M^\ast$ has the same ground set $\{1,2,3,4,5\}$.
    The bases of $M^\ast$ are $$\{24,25,34,35,45\}.$$
    The cocircuits of $M$ are $\{1,23,245,345\}$, hence $1$ is a coloop of $M$ and the elements~$2$ and~$3$ are in series.
    Finally, $M$ is not cosimple, since it has a coloop and a pair of series elements.
\end{example}
In all that follows, we will say that a mod~$2$ map is cosimple when its associated binary matroid is.

Choi and Park proved a statement equivalent to the following proposition in~\cite{Choi-Park2017WedgeOperationsTorus}.
\begin{proposition}\label{prop:seed_cosimple}
    Let $K$ be a mod~$2$ toric-colorable PL~sphere without ghost vertices and $\lambda^\R$ be a mod~$2$ characteristic map over $K$.
    If two vertices $v_1$ and $v_2$ of $K$ are in series with respect to the associated binary matroid $M_{\lambda^\R}$, then the pair $\{v_1, v_2\}$ is either a pair of wedged vertices or a pair of suspended vertices.
\end{proposition}

By convention, the boundary of an interval is considered as a suspended seed.
\begin{corollary}\label{cor:cosimple}
    Let $K$ be a nonsuspended mod~$2$ toric-colorable seed without ghost vertices and $\lambda^\R$ be a mod~$2$ characteristic map over $K$.
    Then, the associated binary matroid $M_{\lambda^\R}$ is cosimple.
\end{corollary}
\begin{proof}
    Since $K\subseteq M_{\lambda^\R}$ and $K$ has dimension $n-1$, every facet of $K$ is a basis of $M_{\lambda^\R}$.
    Hence, $M_{\lambda^\R}$ has no coloop since no vertex of a PL~sphere without ghost vertices is contained in every facet.
  
    If two vertices were in series, then by Proposition~\ref{prop:seed_cosimple} they would be either wedged vertices or suspended vertices.
    This is impossible because $K$ is a nonsuspended seed.
    Hence $M_{\lambda^\R}$ is cosimple.
\end{proof}

If $M$ has rank $n$ on a ground set of size $m$, its \emph{corank} is $m-n$.

\begin{lemma}\label{lem:finiteness_cosimple_matroid}
    There is a finite number of cosimple binary matroids of fixed corank~$p$.
    In particular, their ground sets have at most $2^p - 1$ elements.
    Up to isomorphism, they correspond to the equivalence classes of projective binary codes of length $m$ and dimension $p$.
    Their numbers are listed in Tables~\ref{tab:cosimple_matroid_small_rho_2_3_5} and~\ref{tab:cosimple_matroid_small_rho_5}.
\end{lemma}
\begin{proof}
    The dual of a cosimple binary matroid is simple.
    Hence a representing matrix of the dual matroid has $m$ nonzero and pairwise distinct columns in $\Z_2^p$.
    Therefore $m\le 2^p-1$.
    
    Conversely, every full-row-rank $p \times m$ binary matrix with nonzero and pairwise distinct columns represents a simple binary matroid of rank~$p$, and its dual is a cosimple binary matroid of corank~$p$.
    In rank $p$, such matrices are considered up to row operations and column permutations, which is the same equivalence used for projective binary codes of length $m$ and dimension $p$.
\end{proof}

\begin{table}[h]
    \centering
    \begin{tabular}{*{16}{c}}
        \toprule[1.5pt]
        $m$ &1&2&3&4&5&6&7&8&9&10&11&12&13&14&15 \\ \midrule[.1pt]
        $p=2$ & & 1 & 1\\ 
        $p=3$ & & & 1 & 2 & 1 &1&1\\ 
        $p=4$ & & &   & 1 & 3 & 4 & 5 & 6 & 5 & 4 & 3 & 2 & 1 & 1 & 1 \\ \bottomrule[1.5pt]
    \end{tabular}
    \medskip

    \caption{Number of cosimple binary matroids of small corank $p=2,3,4$ on~$m$ elements, up to isomorphism (see~\url{https://www.mathe2.uni-bayreuth.de/frib/codes/tables_8.html}).\label{tab:cosimple_matroid_small_rho_2_3_5}}
    
\end{table}

\begin{table}[h]
    \centering
    \begin{tabular}{*{16}{c}}
        \toprule[1.5pt]
        $m$& 5&6&7&8&9&10&11&12&13&14&15 & 16 & 17 & 18 \\ \midrule[.1pt]
        $p=5$ & 1 & 4 & 8 & 15 & 29 & 46 & 64 & 89 & 112 & 128 & 144 & 145 &129 &113\\\midrule[1.5pt]
        $m$ &19&20&21&22&23&24&25&26&27&28&29&30&31&  \\\midrule[.1pt]
        $p=5$  & 91 & 67 & 50 & 34 & 21 & 14 & 9 & 5 & 3 & 2 & 1&1 & 1&\\
        \bottomrule[1.5pt]
    \end{tabular}
    \medskip

    \caption{Number of cosimple binary matroids of corank $p=5$ on $m$ elements, up to isomorphism (see~\url{https://www.mathe2.uni-bayreuth.de/frib/codes/tables_8.html}).\label{tab:cosimple_matroid_small_rho_5}}
\end{table}

By Lemma~\ref{lem:finiteness_cosimple_matroid} and Proposition~\ref{prop:seed_cosimple}, we recover the following result of Choi and Park.
\begin{proposition}[\cite{Choi-Park2017WedgeOperationsTorus}]
    The number of mod~$2$ toric-colorable seeds with a fixed Picard number $p$ is finite.
\end{proposition}

We will use the following for initializing our final database.
\begin{lemma}\label{lemma:IDCM_suspension}
    Let $K$ be a mod~$2$ toric-colorable seed on $[m]$ of dimension $n-1$ and of Picard number $p=m-n$. If $K$ is included in a cosimple binary matroid, then the suspension of $K$ is a mod~$2$ toric-colorable seed also included in a cosimple binary matroid.
\end{lemma}
\begin{proof}
    Let $K\subseteq M$, for $M$ a cosimple binary matroid on $[m]$, and let $B$ be a facet of $K$, hence a basis of~$M$, which is also a facet of $K$. The representing matrices of~$M$ and of its dual are respectively
    \[
        \begin{bNiceArray}{c|c}[first-row]
             {\scriptstyle B} & {\scriptstyle [m]\setminus B} \\
            \bI_n & A
        \end{bNiceArray},\quad \text{and} \quad
         \begin{bNiceArray}{c|c}[first-row]
            {\scriptstyle B} & {\scriptstyle [m]\setminus B} \\
            A^\top & \bI_{m-n}
        \end{bNiceArray},
    \]
    and since $M$ is cosimple, all the columns of $\begin{bmatrix}
         A^\top & \bI_{m-n}
    \end{bmatrix}$ are distinct.
    The suspension of $K$ with suspended vertices $m+1$ and $m+2$ is the join $K\ast \partial I$, for $\partial I=\{\emptyset,\{m+1\},\{m+2\}\}$ on $\{m+1,m+2\}$, and is a seed.
    A mod~$2$ characteristic map over the join of two simplicial complexes is given by~\cite[Lemma~3.1]{Choi-Park2016IJM}, and in our case, it is represented by
    \[
        \Lambda=\begin{bNiceArray}{c|c|cc}[first-row]
             {\scriptstyle B} & {\scriptstyle [m]\setminus B} & {\scriptstyle m+1} & {\scriptstyle m+2}\\
            \bI_n & A & \mathbf{0} & \mathbf{x}\\
            \mathbf{0} & \mathbf{a}^\top & 1 & 1
        \end{bNiceArray},
    \]
    for $\mathbf{x}$ and $\mathbf{a}$ some vectors, since the only mod~$2$ characteristic map over $\partial I$ is $\begin{bNiceArray}{cc}[first-row]
             {\scriptstyle m+1} & {\scriptstyle m+2}\\1 & 1
        \end{bNiceArray}$.
    A Gale dual of $\Lambda^\R$ is 
    \[
        {\Lambda}^*=\begin{bNiceArray}{c|c|cc}[first-row]
             {\scriptstyle B} & {\scriptstyle [m]\setminus B} & {\scriptstyle m+1} & {\scriptstyle m+2}\\
            A^\top  & \bI_{m-n} & \mathbf{a} & \mathbf{0}\\
            \mathbf{x}^\top & \mathbf{0}  & 1 & 1
        \end{bNiceArray}.
    \]
    Taking $\mathbf{a}$ being not a column of $A^\top$, for instance a vector of the canonical basis of $\Z_2^{m-n}$ which is not a column of $A^\top$ since $M$ is cosimple, we obtain that all the columns of ${\Lambda}^*$ are distinct, that is, the binary matroid associated to $\Lambda$ is cosimple.
\end{proof}

\begin{corollary}\label{cor:non_IDCM_colorable_seeds}
    Let $K$ be a mod~$2$ toric-colorable seed of Picard number $p$.
    Then, there exists a cosimple binary matroid of corank $p$ and on the same vertex set as $K$ if and only if $K$ is not the boundary of cross-polytopes.
\end{corollary}
\begin{proof}
    The ``only if" direction is well-known.
    See, for example, \cite{Choi2008}.
    For the ``if" part, by Corollary~\ref{cor:cosimple}, every characteristic map over a nonsuspended seed is cosimple. 
    Hence, by Lemma~\ref{lemma:IDCM_suspension}, all iterated suspensions of a seed which is not the boundary of a cross-polytope support cosimple characteristic maps.
\end{proof}

\section{A Gray code technique for enumerating weak pseudomanifolds}

\subsection{Recollection}
We first recall several properties of PL~spheres used in~\cite{Choi-Jang-Vallee2025} for the enumeration of toric-colorable seeds of Picard number~$4$.
The first property we use is that PL~spheres are \emph{weak pseudomanifolds}, which are pure simplicial complexes having each ridge contained in exactly two facets.
In turn, enumerating weak pseudomanifolds is done by linear algebra as follows.

Let $M$ be a pure simplicial complex, $\{\sigma_1,\ldots,\sigma_F\}$ be its set of facets, and $\{\tau_1,\ldots,\tau_R\}$ be its set of ridges. 
The \emph{ridge-facet incidence matrix} of $M$ is the matrix $A(M)\in{\{0,1\}}^{R\times F}$ given by   
$$
    {A(M)}_{rj} =
        \begin{cases}
            1 & \text{if }\tau_r\subseteq \sigma_j,\\
            0 & \text{if }\tau_r\not\subseteq \sigma_j.
        \end{cases}
$$

A pair $(M,K)$ of pure simplicial complexes on the same finite set is called a \emph{pure relative pair} if $M$ and $K$ have the same dimension and $K\subseteq M$.

\begin{example}
    Let $K$ be a PL~sphere and let $\lambda^\R$ be a mod~$2$ characteristic map over $K$. 
    Then, $K$ is a weak pseudomanifold, and $(M_{\lambda^\R},K)$ is a pure relative pair.
\end{example}

Let $(M,K)$ be a pure relative pair and let $\{\sigma_1,\ldots,\sigma_F\}$ be the set of facets of~$M$.
The \emph{facet indicator vector of $K$ in $M$} is denoted by $\chi^K\in{\{0,1\}}^F$ and is given by
$$
    {(\chi^K)}_j = 
        \begin{cases}
            1 & \text{if }\sigma_j\text{ is a facet of }K,\\
            0 & \text{if }\sigma_j\text{ is not a facet of }K.
        \end{cases}
$$

We first recall several properties of PL~spheres used in~\cite{Choi-Jang-Vallee2024} for the enumeration of toric-colorable seeds of Picard number~$4$.
\begin{proposition}[\cite{Choi-Jang-Vallee2024}]\label{proposition:mod_2_kernel}
    Let $(M,K)$ be a pure relative pair.
    Then $K$ is a weak pseudomanifold if and only if the integer vector $A(M)\chi^K$ has only coordinates in~$\{0,2\}$.
    In particular, $\chi^K\in\ker_2 A(M)$.
\end{proposition}
Proposition~\ref{proposition:mod_2_kernel} gives a direct algorithm for computing all weak pseudomanifold subcomplexes of a given pure simplicial complex $M$ with $F$ facets.
In our applications, $M$ will be a cosimple binary matroid.
The algorithm proceeds as follows.
\begin{enumerate}
    \item Compute a mod~$2$ basis $B = \begin{bmatrix} b_1& \cdots & b_d \end{bmatrix}$ of the kernel of the ridge-facet incidence matrix $A(M)$.
    \item For every vector $X\in \Z_2^d$, compute $\chi = BX$ in $\Z_2^F$. 
    If the integer vector $A(M)\chi$ has only entries in $\{0,2\}$, store $\chi$.
\end{enumerate}

In order to improve the total computation time of the enumeration of all $\chi^K=BX$, we will use a method called Gray code enumeration for $X$ as follows.
For $d\geq 1$, a \emph{Gray code enumeration} of $\{0,1\}^d$ is an ordering $X^0,X^1,\ldots,X^{2^d-1}$ of the elements of $\{0,1\}^d$ such that $X^{\kappa-1}$ and $X^\kappa$ differ in exactly one coordinate for every $1\le \kappa\le 2^d-1$.

\begin{example}
    The most commonly used Gray code enumeration is the \emph{binary-reflected Gray code}, defined by
    $$
        X^\kappa = \kappa \oplus (\kappa \gg 1), \qquad 0 \leq \kappa \leq 2^d-1,
    $$
    where $\oplus$ denotes the bitwise XOR operation, and $\gg$ represents a bit shift.
    In this encoding, the coordinate that changes between $X^{\kappa-1}$ and $X^\kappa$ is the active bit of $X^{\kappa-1}\oplus X^\kappa$.
\end{example}

We use a function $\texttt{Gray}$ associated with a fixed Gray code enumeration of $\{0,1\}^d$.
For $1\le \kappa\le 2^d-1$, the value $\texttt{Gray}[\kappa]$ is the coordinate that changes between $X^{\kappa-1}$ and $X^\kappa$.

We prepare a Gray code enumeration of weak pseudomanifolds in the mod~$2$ kernel of a matrix.
\begin{algorithm}
    \caption{Gray code enumeration of weak pseudomanifolds in the mod~$2$ kernel of a matrix.\label{algorithm:Gray_code}}
    \KwData{%
            - A matrix $A\in{\{0,1\}}^{R\times F}$ in sparse column form: $A_{\scol}[j]$ is the list of row indices $r$ with $A_{r j}=1$,\newline
            - A matrix $B\in{\{0,1\}}^{F\times d}$ whose columns $B[1],\ldots,B[d]$ form a basis of the mod~$2$ kernel of~$A$.}
    \KwResult{The set $\texttt{wkPsdMfd}$ of all nonzero $y\in\ker_2(A)$ such that the integer vector $Ay\in{\{0,2\}}^R$.}

    $y \gets \mathbf{0}\in{\{0,1\}}^F$\;
    $\rho \gets \mathbf{0}\in \Z^R$\;
    $\texttt{wkPsdMfd} \gets \emptyset$\;

    \For{$\kappa \gets 1$ \KwTo$2^d - 1$}{
        $\beta \gets \texttt{Gray}[\kappa]$
        \tcp*{index of the bit flipped at step~$\kappa$}
        $y \gets y \oplus B[\beta]$\;
        \For{$j\in\supp(B[\beta])$}{
            \lFor{$r\in A_{\scol}[j]$}{$\rho[r] \gets \rho[r] + (2\,y[j]-1)$}
        }
        \lIf{$\rho\in{\{0,2\}}^R$}{$\texttt{wkPsdMfd} \gets \texttt{wkPsdMfd}\cup\{y\}$}
    }
    \Return$\texttt{wkPsdMfd}$\;
\end{algorithm}

\begin{proposition}\label{prop:Gray_code_correctness}
    Algorithm~\ref{algorithm:Gray_code} returns exactly the nonzero vectors $y\in\ker_2(A)$ such that the integer vector $Ay$ belongs to $\{0,2\}^R$.
\end{proposition}
\begin{proof}
    The Gray code enumeration visits every nonzero vector $X\in\{0,1\}^d$ exactly once. Since the columns of $B$ form a basis of $\ker_2(A)$, the vectors $y=BX$ are precisely the nonzero elements of $\ker_2(A)$.

    We maintain the invariant that $\rho=Ay$ over $\Z$. 
    Initially, $y=0$ and $\rho=0$. 
    At a Gray step, only the coordinates $j\in\supp(B[\beta])$ are toggled. 
    If the new value of~$y[j]$ is $1$, the contribution of column $j$ to each incident row increases by $1$.
    If the new value is $0$, it decreases by $1$. 
    Thus the update
    $$
        \rho[r]\gets \rho[r]+(2y[j]-1)
    $$
    preserves the invariant $\rho=Ay$. 
    Therefore the test $\rho\in\{0,2\}^R$ is exactly the weak pseudomanifold test.
\end{proof}
This reduces the total complexity of the enumeration process as follows.
\begin{proposition}
    Let $M$ be a pure simplicial complex of dimension $n-1$, and suppose that the dimension of the mod~$2$ kernel of $A(M)$ is $d$.
    Then each column of $A(M)$ has exactly $n$ nonzero entries.
    A classical enumeration uses $\cO(F(d+n)2^d)$ elementary operations, while the Gray code enumeration uses $\cO(Fn2^d)$ atomic bitwise operations.
\end{proposition}
\begin{proof}
    For $0\leq \kappa \leq 2^d-1$, consider the current vector ${y_\kappa}=BX_\kappa$.
    
    \textit{Classical enumeration.} Computing ${y_{\kappa+1}} = BX_{\kappa+1}$ from scratch requires $\cO(F \cdot d)$ bitwise operations. 
    The integer image $A{y_{\kappa+1}}$ is then computed from scratch by summing over all $j\in\supp({y_{\kappa+1}})$.
    Since each column $A_{\scol}[j]$ has exactly $n$ nonzero entries, this costs $\cO(F \cdot n)$ operations.
    Since we check whether $\rho \in \{0,2\}^R$ at each step, the total per-step cost is thus $\cO(F(d+n)+R)$, giving $\cO((F(d+n)+R)\cdot 2^d)$ overall.
    Since $R\le nF$, the cost is $\cO(F(d+n)2^d)$.

    \textit{Gray code enumeration.} For $1\leq \kappa < 2^d$, $X_\kappa$ and $X_{\kappa+1}$ differ only at a single index $\beta\in[d]$.
    Thus, we have
    $$
        {y_{\kappa+1}} = {y_\kappa} \oplus B[\beta],
    $$
    which requires $\cO(F)$ bitwise operations. 
    Moreover, only the coordinates $j\in\supp(B[\beta])$ of $\chi^K$ are toggled, so only the row sums $\rho[r]$ for rows $r\in A_{\scol}[j]$ need updating.
    Since each column $A_{\scol}[j]$ has exactly $n$ nonzero entries and $|\supp(B[\beta])|\leq F$, updating the row sums costs at most $\cO(F\cdot n +R) =  \cO(Fn)$ operations per step. 
    The total cost is therefore $\cO(F n\cdot 2^d)$.
\end{proof}

In the Picard number $4$ computations, the values of $n$ range from $2$ to $11$, while the values of $d$ range from $10$ to $30$.
Thus the Gray code method removes the repeated matrix-vector multiplication by $B$.
In our implementation, for the Picard number four instances considered in \cite{Choi-Jang-Vallee2024}, this gives about a fourfold speedup.

\begin{remark}
    The method also parallelizes naturally.
    Fix the first $w$ coordinates and choose a head vector $X^{\head}\in\{0,1\}^w$.
    Assign one head vector to each of the $2^w$ workers.
    Each worker initializes $y$ and $\rho$ from this head vector and then runs a Gray code enumeration on the remaining cube $\{0,1\}^{d-w}$.
\end{remark}

It remains to adapt this enumeration algorithm to the dynamic programming method.
For this purpose, we introduce a few more definitions.
A pure simplicial complex is a \emph{weak pseudomanifold with boundary} if each ridge is contained in at most $2$ facets.

\begin{example}
    PL~balls are weak pseudomanifolds with boundary.
\end{example}

\subsection{Gray code enumeration for weak pseudomanifolds with boundary}
We now describe the extension step used in our dynamic programming scheme.
It is adapted from the final enumeration step of~\cite{Choi-Jang-Vallee2024}.

Given a pure relative pair $(M,L)$ with $L$ a weak pseudomanifold with boundary, the algorithm outputs all pure relative pairs $(M,K)$ such that $L\subseteq K$ and $K$ is a weak pseudomanifold.
Let $\chi^L\in\{0,1\}^F$ denote the facet indicator vector of the facets of $L$.
We say that a ridge at index $r$ is \emph{incomplete} in $L$ if $(A(M)\chi^L)_r=1$.
It is called \emph{full} if $(A(M)\chi^L)_r=2$.
Here $A(M)\chi^L$ is computed over $\Z$.

We now describe all steps of Algorithm~\ref{algorithm:Gray_code_fixed_facets}.

\noindent\textbf{Phase 1. Unit propagation.}
We iteratively derive two sets of constraints on the facets of $M$ that are not facets of $L$.
We maintain a vector $\chi^+\in\{0,1\}^F$ of \emph{mandatory} facets (initialized to $\chi^L$) and a vector $\chi^-\in\{0,1\}^F$ of \emph{forbidden} facets (initialized to $\mathbf{0}$), and repeat the following until no change occurs.
For each ridge $r$, let $s = |\{j\in A(M)[r] \mid \chi^+[j]=1\}|$ be the number of mandatory facets already covering $r$, and let $\cF_r = \{j\in A(M)[r] \mid \chi^+[j]=0, \chi^-[j]=0\}$ be the set of remaining free facets incident to $r$.
\begin{itemize}
    \item If $s > 2$, the current partial solution is infeasible and we return $\emptyset$.
    \item If $s = 2$, every free facet incident to $r$ must be forbidden: set $\chi^-[j]\gets 1$ for all $j\in \cF_r$.
    \item If $s=1$ and $|\cF_r|=0$, the current partial solution is infeasible and we return $\emptyset$.
    \item If $s = 1$ and $|\cF_r| = 1$, the unique free facet is forced: set $\chi^+[\cF_r[1]]\gets 1$.
    \item If $s = 0$ and $|\cF_r| = 1$, adding $\cF_r[1]$ would leave $r$ incomplete with no further option; set $\chi^-[\cF_r[1]]\gets 1$.
\end{itemize}

\noindent\textbf{Phase 2. Constrained column echelon form.}
Let $B\in\{0,1\}^{F\times d}$ be a matrix whose columns form a mod~$2$ basis of $\ker_2(A(M))$.
We compute a reduced column echelon form $\hat B[1],\ldots,\hat B[d]$ of $B$ over $\Z_2$ with pivots $p_1,\ldots,p_d$.
The pivot search first uses rows in $\supp(\chi^+)$, then rows in $\supp(\chi^-)$, and finally the remaining rows.
This yields a partition of the column indices into 
\begin{align*}
    \cI^+ &= \{i\in[d]\mid p_i\in\supp(\chi^+)\}, \\
    \cI^- &= \{i\in[d]\mid p_i\in\supp(\chi^-)\}, \text{ and }\\
    \cI^\circ  & = [d]\setminus(\cI^+\cup \cI^-).  
\end{align*}

Because all constrained rows are processed before unconstrained rows, any constrained row that is not selected as a pivot has zero entries in all remaining free generators. 
Hence its value is already determined once the coefficients corresponding to constrained pivot rows are fixed.

\noindent\textbf{Phase 3. Initialization from forced basis vectors.}
Since $\hat B$ restricted to the pivot rows is the identity matrix, the generator $\hat B[i]$ is the unique way of including the mandatory facet $p_i$ for $i\in \cI^+$, and must be excluded to avoid the forbidden facet $p_i$ for $i\in \cI^-$.
We therefore initialize
$$
    y \gets \bigoplus_{i\in \cI^+} \hat B[i],
$$
and discard the generators indexed by $\cI^-$ entirely, reducing the search space by a factor of $2^{|\cI^+|+|\cI^-|}$.
We then verify that $y$ is consistent: if any mandatory facet is absent from $\supp(y)$, or any forbidden facet belongs to $\supp(y)$, we return $\emptyset$.
Indeed, for every $y'$ obtained in the Gray code enumeration, the free generators indexed by $\cI^\circ$ have zero entries on all constrained rows,
that is, on $\supp(\chi^+) \cup \supp(\chi^-)$.
Hence
$$
    y'|_{\supp(\chi^+) \cup \supp(\chi^-)} = y|_{\supp(\chi^+) \cup \supp(\chi^-)}.
$$
Consequently $y'$ satisfies two conditions 
$$
    \supp(\chi^+) \subset \supp(y') \quad \text{ and } \quad \supp(y') \cap \supp(\chi^-) = \varnothing
$$
if and only if $y$ satisfies the same two conditions.

\noindent\textbf{Phase 4. Gray code enumeration.}
It remains to enumerate all extensions of $y$ by subsets of the free generators $\{\hat B[i]\}_{i\in \cI^\circ}$, i.e., all elements of the form
\[
    y \oplus \bigoplus_{i\in \cI^\circ} x_i \hat B[i], \qquad x_i\in\{0,1\},
\]
that additionally satisfy $A(M)y'\in\{0,2\}^R$.
As previously, rather than checking each of the $2^{|\cI^\circ|}$ candidates from scratch, we maintain the row-sum vector $\rho[r] = \sum_{j\in A(M)[r]} y[j]$ incrementally.
All $2^{|\cI^\circ|}$ candidates, including the initial one $y$, are checked. 
After testing $y$, the remaining candidates are traversed in Gray code order, so that consecutive candidates differ by the flip of a single free generator $\hat B[\beta]$.
Each flip updates $\rho$ in time proportional to $n|\supp(\hat B[\beta])| \leq nF$.
A candidate $y'$ is a valid extension if and only if $\rho\in\{0,2\}^R$, and we check this after every complete update of $\rho$.
Thus, recalling that $R\leq nF$, the enumeration phase costs at most $\cO\left(nF\cdot 2^{|\cI^\circ|}\right)$.

\begin{algorithm}
    \caption{$\texttt{GrayCodeExtend}$, Gray code enumeration of weak pseudomanifolds extending a given weak pseudomanifold with boundary.\label{algorithm:Gray_code_fixed_facets}}
    \KwData{%
    - A matrix $A\in{\{0,1\}}^{R\times F}$ in sparse row and sparse column forms: $A_{\srow}[r]$ is the list of column indices $j$ with $A_{r j}=1$, and $A_{\scol}[j]$ is the list of row indices $r$ with $A_{r j}=1$,\newline
            - A vector $\chi^+ =\chi^L\in{\{0,1\}}^F$ of mandatory columns, where $L (\neq \emptyset)$ is the prescribed weak pseudomanifold with boundary,\newline
            - A matrix $B\in{\{0,1\}}^{F\times d}$ whose columns $B[1],\ldots,B[d]$ form a basis of the mod~$2$ kernel of~$A$.}
    \KwResult{The set $\texttt{wkPsdMfd}$ of all $y\in\ker_2(A)$ such that the integer vector $Ay$ belongs to $\{0,2\}^R$, $\supp(\chi^+)\subseteq\supp(y)$, and $\supp(y)\cap\supp(\chi^-)=\emptyset$.}

    \tcc{Phase 1 - Unit propagation}
    $\chi^- \gets \mathbf{0} \in {\{0,1\}}^F$\;
    \Repeat{\textnormal{no change to $\chi^+$ or $\chi^-$}}{
        \For{each row $r \in [R]$}{
            $s \gets \bigl|\{j \in A_{\srow}[r] : \chi^+[j]=1\}\bigr|$\;
            $\cF_r \gets \{j \in A_{\srow}[r] : \chi^+[j]=0,\; \chi^-[j]=0\}$\;
            \lIf{$s > 2$}{\Return$\emptyset$}
            \lIf{$s = 2$}{set $\chi^-[j]\gets 1$ for all $j\in \cF_r$}
            \lIf{$s = 1$ and $|\cF_r|=0$}{\Return$\emptyset$}
            \lIf{$s = 1$ and $|\cF_r|=1$}{set $\chi^+[\cF_r[1]]\gets 1$}
            \lIf{$s = 0$ and $|\cF_r|=1$}{set $\chi^-[\cF_r[1]]\gets 1$}
        }
    }
   \tcc{Phase 2 - Constrained column echelon form}
    Compute a column echelon form $\hat{B}[1],\ldots,\hat{B}[d]$ of $B$ with pivots $p_1,\ldots,p_{d}$,
    prioritising pivots in $\supp(\chi^+)$ first, then pivots in $\supp(\chi^-)$, and placing free columns last\;
    \tcc{Phase 3 - Initialization from forced basis vectors}
    $y \gets \mathbf{0}\in{\{0,1\}}^F$\;
    $\cI^\circ \gets [\,]$ \tcp*{ordered list of free basis indices}
    \For{$i \gets 1$ \KwTo$d$}{
        \lIf{$\chi^+[p_i]=1$}{$y \gets y \oplus \hat{B}[i]$}
        \lElseIf{$\chi^-[p_i]=0$}{append $i$ to $\cI^\circ$}
    }
    \lIf{$\exists j\in[F]\text{ such that }\chi^+[j]=1\text{ and }y[j]=0$}{\Return$\emptyset$}
    \lIf{$\exists j\in[F]\text{ such that }\chi^-[j]=1\text{ and }y[j]=1$}{\Return$\emptyset$}

    \tcc{Phase 4 - Gray code enumeration with incremental row sums} 
    $\texttt{wkPsdMfd}\gets$ output of an adaptation of Algorithm~\ref{algorithm:Gray_code} that first tests the initial vector $y$, and then toggles only the columns $\hat B[i]$ with $i \in \cI^\circ$ in Gray code order.
    
    \Return$\texttt{wkPsdMfd}$\;
\end{algorithm}

\begin{proposition}\label{prop:GrayCodeExtend_correctness}
    Algorithm~\ref{algorithm:Gray_code_fixed_facets} returns exactly the vectors $y\in\ker_2(A)$ such that
    $$
        Ay\in\{0,2\}^R,\qquad \supp(\chi^L)\subseteq\supp(y),
    $$
    and hence exactly the weak pseudomanifold subcomplexes of $M$ extending $L$.
\end{proposition}
\begin{proof}
    The unit propagation rules in Phase~1 are sound consequences of the condition that every row sum must be either $0$ or $2$. 
    Hence they do not exclude any valid extension of the prescribed complex $L$.
    The column operations used to obtain $\hat B$ only replace the chosen basis of $\ker_2(A)$ by another basis of the same vector space. 
    Since the constrained pivot rows form an identity matrix, the coefficients indexed by $\cI^+$ and $\cI^-$ are forced, while the coefficients indexed by $\cI^\circ$ parametrize all remaining candidates satisfying the forced constraints.
    The Gray code traversal visits all choices of these free coefficients exactly once. 
    Throughout the traversal, the maintained vector $\rho$ is equal to $Ay$ over $\Z$, so the test $\rho\in\{0,2\}^R$ is exactly the weak pseudomanifold condition by Proposition~\ref{proposition:mod_2_kernel}. 
    Therefore Algorithm~\ref{algorithm:Gray_code_fixed_facets} returns precisely the weak pseudomanifold subcomplexes of $M$ extending $L$.
\end{proof}

\section{Enumeration via binary matroids}

\subsection{Cosimple binary matroids and the contraction category}

By Corollary~\ref{cor:non_IDCM_colorable_seeds}, every mod~$2$ toric-colorable seed which is not the boundary of a cross-polytope is a subcomplex of a cosimple binary matroid of the same Picard number. Consequently, the enumeration of mod~$2$ toric-colorable seeds via cosimple binary matroids reduces to the finite collection of cosimple binary matroids of corank~$p$.
Our enumeration algorithm recursively reconstructs simplicial complexes from their vertex links. Since links correspond to contractions of binary matroids, we organize the cosimple binary matroids of corank~$p$ into the following contraction category, which serves as the indexing structure for the dynamic programming algorithm developed in the next subsection.

The \emph{contraction category} of a matroid $M$ is the category $\mathsf{Contr}(M)$ described below.
Its objects are the contraction minors of $M$, taken up to isomorphism.
For two objects $N$ and $N'$, a generating morphism $N\to N'$ is given by a single-element contraction $N'\cong N/v$, together with a chosen relabelling of the ground set of $N'$ to that of $N/v$.
All morphisms are generated by these single-contraction morphisms under composition.

From Proposition~\ref{proposition:contraction_matroids}, we can enumerate all cosimple binary matroids of a given corank~$p$ by starting with the unique one having $2^p -1$ elements in its ground set.
We denote this matroid by $X(\Z_2^p)^*$ as it is the dual matroid to the so-called \emph{mod~$2$ universal complex of rank $p$} in~\cite{Baralic-Vavpetic-Vucic2023}, denoted by $X(\Z_2^p)$.

We encode the contraction category of $X(\Z_2^p)^*$ using the following data.
\begin{itemize}
    \item \texttt{ObjBinMat} is a dictionary of lists of matroids.
        Its list at key $m$ contains one representative of each isomorphism class of contraction minors of $X(\Z_2^p)^*$ with ground set of size $m$.

    \item \texttt{MorBinMat} stores the one-step contraction morphisms.
        For $v\in[m]$, the entry $\texttt{MorBinMat}[m][i][v]$ is a pair $(i',\phi)$ consisting of an index of an element in \texttt{ObjBinMat}[m-1] together with a bijection $[m-1]\to [m]\setminus\{v\}$. 
        More precisely, $\texttt{ObjBinMat}[m-1][i']$ is isomorphic to $\texttt{ObjBinMat}[m][i]/v$, and
        $$
            \phi\colon [m-1]\to [m]\setminus\{v\}
        $$
        is the chosen relabelling giving this isomorphism.
\end{itemize}

Algorithm~\ref{algorithm:generate_contraction_category} constructs the contraction category of $X(\Z_2^p)^*$.

\begin{algorithm}
\KwData{The matroid $X(\Z_2^p)^*$.}
\KwResult{\texttt{ObjBinMat} and \texttt{MorBinMat}.}
\tcc{We first create all the objects} $\texttt{ObjBinMat}[2^p-1]\gets[X(\Z_2^p)^*]$\;
\For{$m=2^p-1,2^p-2,\ldots,p+1$}{
    $\texttt{ObjBinMat}[m-1]\gets[]$\;
    \For{$M\in \texttt{ObjBinMat}[m]$}{
        \For{every vertex $v$ of $M$}{
            \If{$v$ is not a loop of $M$}{
                \If{$M/v$ is not isomorphic to any element of $\texttt{ObjBinMat}[m-1]$}{
                    Append $M/v$ to $\texttt{ObjBinMat}[m-1]$\;
                }
            }
        }
    }
}
\tcc{We now create all the morphisms}
\For{$m=2^p-1,2^p-2,\ldots,p+1$}{
    \For{$i\gets1$ \KwTo$\text{len}(\texttt{ObjBinMat}[m])$}{$M\gets \texttt{ObjBinMat}[m][i]$\;
        \For{every vertex $v$ of $M$}{
            \If{$v$ is not a loop of $M$}{
                $\texttt{MorBinMat}[m][i][v]\gets(i',\phi)$ where $\phi(\texttt{ObjBinMat}[m-1][i'])=M/v$\;
            }
        }
    }
}
\Return$(\texttt{ObjBinMat},\texttt{MorBinMat})$\;
\caption{Construction of the contraction category of $X(\Z_2^p)^*$.\label{algorithm:generate_contraction_category}}
\end{algorithm}

\subsection{Dynamic programming on the contraction category}

The key ingredient of the dynamic programming algorithm is that every weak pseudomanifold can be built up from the link at one of its vertices: this link determines the facets containing that vertex, and the remaining facets are then enumerated by extending this prescribed portion. The following proposition establishes, at the level of binary matroids, the compatibility condition that makes this induction well founded.

\begin{proposition}\label{prop:link_wkpsdmfd}
    Let $M$ be a matroid of rank $n$ and let $K\subseteq M$ be a weak pseudomanifold of dimension $n-1$. 
    If $v$ is a vertex of $K$, the link $K/v$ is a weak pseudomanifold contained in $M/v$.
\end{proposition}
\begin{proof}
    Since $v$ is a vertex of $K$ and $K\subseteq M$, the element $v$ is not a loop of $M$.
    Thus the matroid contraction $M/v$ is well-defined.
    The inclusion $K/v\subseteq M/v$ follows directly from the definitions. 
    Let $\rho$ be a ridge of $K/v$; then $\rho\cup\{v\}$ is a ridge of $K$, so it is contained in exactly two facets $\sigma_1$ and $\sigma_2$ of $K$, both of which contain $v$. 
    The sets $\sigma_1\setminus\{v\}$ and $\sigma_2\setminus\{v\}$ are then the two facets of $K/v$ containing $\rho$, so $K/v$ is a weak pseudomanifold.
\end{proof}

Conversely, let $K\subseteq M$ be a weak pseudomanifold and let $v$ be one of its vertices.
Knowing the link $K/v\subseteq M/v$ determines all facets of $K$ containing $v$, namely the facets of the cone $K/v\ast v$.
Since $K/v$ is a weak pseudomanifold, the cone $K/v\ast v$ is a weak pseudomanifold with boundary.
Thus the link specifies a prescribed weak pseudomanifold with boundary inside $M$, leaving only the remaining facets to be determined.

Algorithm~\ref{algorithm:Gray_code_fixed_facets} can therefore be applied with $(A(M),\chi^+,B(M))$ to enumerate all weak pseudomanifolds $K\subseteq M$ whose support contains $\supp(\chi^+)$.
The prescribed cone $K/v\ast v$ is precisely the input required by Algorithm~\ref{algorithm:Gray_code_fixed_facets}. If
\[
\chi^+ := \chi^{K/v\ast v},
\]
then $\supp(\chi^+)\subseteq\supp(\chi^K)$.
Hence Algorithm~\ref{algorithm:Gray_code_fixed_facets} enumerates all weak pseudomanifolds $K\subseteq M$ extending the prescribed cone.

We now apply the previous extension procedure recursively over the contraction category. To avoid recomputing the same extensions multiple times, we store, for each binary matroid already processed, the collection of all weak pseudomanifolds it contains, using the following data structure: for each $m$ and each index $i$, let
$$
    \begin{aligned}
        \texttt{WkPsdMfd}[m][i] = \bigl\{ \chi^K \mid{}& K\subseteq \texttt{ObjBinMat}[m][i],\\
            &K\text{ is a weak pseudomanifold}   \bigr\}.
    \end{aligned}
$$

The algorithm processes matroids level by level, from the smallest ground set on $m_{\min}$ up to $m_{\max}$, with $m_{\min} \leq m_{\max} \leq 2^p-1$.
We set $m_{\min} = p+2$, since the lower-rank levels do not contribute to any seed without ghost vertices of Picard number~$p$ in the range considered here.

At the base level $m_{\min}$, the matroids are small enough that Algorithm~\ref{algorithm:Gray_code} is applied directly, with no mandatory facets. 

For each larger matroid $M = \texttt{ObjBinMat}[m][i]$ (with $m$ larger than the base level), for every vertex $v$ of $M$ which is not a loop, $M/v$ is cosimple of corank~$p$ by Proposition~\ref{proposition:contraction_matroids}. 
The morphism $\texttt{MorBinMat}[m][i][v] = (i',\phi)$ provides the index $i'$ such that $\texttt{ObjBinMat}[m-1][i']\cong_\phi M/v$, together with the relabelling $\phi\colon[m-1]\xrightarrow{\sim}[m]\setminus\{v\}$.

For each ${y'}\in\texttt{WkPsdMfd}[m-1][i']$, let $K'$ denote the corresponding weak pseudomanifold. Using the relabelling $\phi$, we view $K'$ as a subcomplex of $M/v$ and form the cone $\phi(K')\ast v$, whose facet indicator vector gives the mandatory vector
$$
    \chi^+ := \chi^{\phi(K')\ast v}\in{\{0,1\}}^F.
$$
Running Algorithm~\ref{algorithm:Gray_code_fixed_facets} on $\chi^+$ for every non-loop vertex~$v$ and every such $K'$, and collecting the results, is precisely Algorithm~\ref{algorithm:dynamic_programming} below, whose correctness is established in Proposition~\ref{prop:dynamic_programming_correctness}.

\begin{algorithm}
    \caption{Dynamic programming enumeration of weak pseudomanifolds of Picard number~$p$.\label{algorithm:dynamic_programming}}
    \KwData{The contraction category $(\texttt{ObjBinMat}, \texttt{MorBinMat})$ of $X(\Z_2^p)^*$ from Algorithm~\ref{algorithm:generate_contraction_category}.}
    \KwResult{For every $m\geq m_{\min}$ and $i$, the set $\texttt{WkPsdMfd}[m][i]$ of facet indicator vectors of all weak pseudomanifolds contained in $\texttt{ObjBinMat}[m][i]$.}

    \For{$i\gets 1$ \KwTo$|\texttt{ObjBinMat}[m_{\min}]|$\hfill\tcp{Base case}}{
        $M \gets \texttt{ObjBinMat}[m_{\min}][i]$\;
        $\texttt{WkPsdMfd}[m_{\min}][i] \gets$ output of Algorithm~\ref{algorithm:Gray_code} on $\bigl(A(M),\, B(M)\bigr)$\;
    }
    \For{$m \gets m_{\min}+1$ \KwTo $m_{\max}$\hfill\tcp{Inductive step}}{
        \For{$i \gets 1$ \KwTo$|\texttt{ObjBinMat}[m]|$}{
            $\texttt{WkPsdMfd}[m][i] \gets \emptyset$\;
            $M \gets \texttt{ObjBinMat}[m][i]$\;
            Let $\{V_1,\ldots,V_k\}$ be the orbits of the non-loop vertices of $M$ under $\operatorname{Aut}(M)$,with representatives $v_1,\ldots,v_k$\;
            \For{$j=1,\ldots,k$}
            {
            $(i',\phi) \gets \texttt{MorBinMat}[m][i][v_j]$
            \tcp*{$\phi\colon[m-1]\xrightarrow{\sim}[m]\setminus\{v_j\}$, $\texttt{ObjBinMat}[m-1][i']\cong_\phi M/v_j$}
            $S_j \gets \emptyset$\;
            \For{${y'} \in \texttt{WkPsdMfd}[m-1][i']$}{
                $\chi^+ \gets \chi^{\phi(K')\ast v_j} \in {\{0,1\}}^F$
                \tcp*{mandatory facets}
                $S_j \gets S_j \cup \texttt{GrayCodeExtend}\bigl(A(M),\, \chi^+,\, B(M)\bigr)$\;
            }
            $\texttt{WkPsdMfd}[m][i] \gets \texttt{WkPsdMfd}[m][i] \cup S_j$\;
            \For{$v \in V_j\setminus\{v_j\}$}{
                    Let $\psi\in\operatorname{Aut}(M)$ satisfy $\psi(v_j)=v$\;
                    $\texttt{WkPsdMfd}[m][i] \gets
                        \texttt{WkPsdMfd}[m][i]
                        \cup \{\psi(\chi)\mid \chi\in S_j\}$\;
                }
        }}
    }
    \Return$\texttt{WkPsdMfd}$\;
\end{algorithm}

\begin{proposition}\label{prop:dynamic_programming_correctness}
    Algorithm~\ref{algorithm:dynamic_programming} computes, for every
    $m_{\min}\leq m\leq m_{\max}$ and every index~$i$, the set of facet
    indicator vectors of all nonempty weak pseudomanifolds
    $K\subseteq\texttt{ObjBinMat}[m][i]$.
\end{proposition}
\begin{proof}
    We proceed by induction on $m$.

    \textit{Base case} ($m=m_{\min}$).
    For each $i$, with $M=\texttt{ObjBinMat}[m_{\min}][i]$, Algorithm~\ref{algorithm:Gray_code}
    applied to $(A(M),B(M))$ returns exactly the facet indicator vectors of all nonempty
    weak pseudomanifolds contained in $M$, by Proposition~\ref{prop:Gray_code_correctness}.

    \textit{Inductive step} ($m>m_{\min}$).
    Suppose the claim holds for $m-1$, and let $M=\texttt{ObjBinMat}[m][i]$, $n=\rank(M)$.

    \textit{Soundness.}
    Every vector added to $\texttt{WkPsdMfd}[m][i]$ arises in one of two ways.
    First, as an output of $\texttt{GrayCodeExtend}(A(M),\chi^+,B(M))$, collected into
    $S_j$, for some orbit representative $v_j$ and some $y'\in\texttt{WkPsdMfd}[m-1][i']$;
    by Proposition~\ref{prop:GrayCodeExtend_correctness} such an output is the facet
    indicator vector of a weak pseudomanifold $K\subseteq M$, nonempty since $\chi^+$
    is itself the nonempty indicator vector of the cone $\phi(K')\ast v_j$ and
    $\supp(\chi^+)\subseteq\supp(\chi^K)$.
    Second, as $\psi(\chi)$ for some $\chi\in S_j$ and $\psi\in\operatorname{Aut}(M)$, in
    the symmetrization step; since $\psi$ sends facets, ridges, and their incidences of
    any subcomplex of $M$ to those of its image, it maps every nonempty weak
    pseudomanifold contained in $M$ to another one, so $\psi(\chi)$ is again the facet
    indicator vector of a nonempty weak pseudomanifold contained in $M$.

    \textit{Completeness.}
    Let $K\subseteq M$ be a nonempty weak pseudomanifold of dimension $n-1$, and let
    $v$ be a non-ghost vertex of $K$ (which exists since $K$ is nonempty); then $v$ is
    not a loop of $M$.
    Let $v_j$ be the representative of the orbit of $v$ under $\operatorname{Aut}(M)$,
    and let $\psi\in\operatorname{Aut}(M)$ satisfy $\psi(v_j)=v$.
    Set $\widetilde K:=\psi^{-1}(K)$; by the argument used in Soundness, $\widetilde K$
    is again a nonempty weak pseudomanifold contained in $M$, and $v_j=\psi^{-1}(v)$ is
    a non-ghost vertex of $\widetilde K$.
    It therefore suffices to show that $\chi^{\widetilde K}$ is added to $S_j$ during the
    processing of representative $v_j$: the symmetrization step of
    Algorithm~\ref{algorithm:dynamic_programming} then adds
    $\psi(\chi^{\widetilde K})=\chi^K$ as well.

    Let $(i',\phi)=\texttt{MorBinMat}[m][i][v_j]$, so
    $\texttt{ObjBinMat}[m-1][i']\cong_\phi M/v_j$.
    By Proposition~\ref{prop:link_wkpsdmfd}, $\widetilde K/v_j$ is a nonempty weak
    pseudomanifold of dimension $n-2$ contained in $M/v_j$ (nonempty since $v_j$ lies
    in a facet of $\widetilde K$), so
    $y':=\chi^{\phi^{-1}(\widetilde K/v_j)}\in\texttt{WkPsdMfd}[m-1][i']$ by the
    inductive hypothesis.

    When the algorithm processes representative $v_j$ and vector $y'$, it forms
    $\chi^+=\chi^{\phi(K')\ast v_j}=\chi^{(\widetilde K/v_j)\ast v_j}$, and
    $\supp(\chi^+)\subseteq\supp(\chi^{\widetilde K})$ as noted following
    Proposition~\ref{prop:link_wkpsdmfd}.
    By Proposition~\ref{prop:GrayCodeExtend_correctness},
    $\texttt{GrayCodeExtend}(A(M),\chi^+,B(M))$ then returns all weak pseudomanifolds
    extending $(\widetilde K/v_j)\ast v_j$, in particular $\chi^{\widetilde K}$, which is
    added to $S_j$.
    By the symmetrization step, $\chi^K=\psi(\chi^{\widetilde K})$ is then also added to
    $\texttt{WkPsdMfd}[m][i]$.

    Since $K$ was arbitrary, the induction is complete.
\end{proof}

\begin{remark}\label{rk:strong_iso}
    The correctness of the orbit reduction in
    Algorithm~\ref{algorithm:dynamic_programming} is established within
    the proof of Proposition~\ref{prop:dynamic_programming_correctness}:
    Soundness covers the vectors produced by the symmetrization step, and
    Completeness shows that for a non-loop vertex $v$ in the orbit $V_j$ of
    representative $v_j$, it suffices to locate $\chi^{\widetilde K}$ in $S_j$ for
    $\widetilde K=\psi^{-1}(K)$, $\psi\in\operatorname{Aut}(M)$ with $\psi(v_j)=v$,
    after which the relabeling step recovers $\chi^K=\psi(\chi^{\widetilde K})$.

    We record in Table~\ref{tab:orbit_savings} the savings achieved by
    this reduction.
    For Picard number~$4$, across all levels $m=7$ to $m=15$,
    the $265$ non-loop vertex--matroid pairs are reduced to $66$
    orbit representatives, giving an overall speedup of $4.02\times$.
    The reduction factor grows monotonically with~$m$, from
    $2.62\times$ at $m=7$ to $15.00\times$ at $m=15$, reflecting
    the fact that cosimple binary matroids of corank~$4$ with fewer
    elements tend to have larger automorphism groups acting more
    transitively on their non-loop vertices.
    In particular, at $m=14$ and $m=15$ every non-loop vertex lies
    in a single orbit, so Algorithm~\ref{algorithm:dynamic_programming}
    performs exactly one call to $\texttt{GrayCodeExtend}$ per link
    at those levels.
    For Picard number~$5$, the orbit reduction is applied at levels
    $m=8$ and $m=9$, where the savings are reported in the lower
    panel of Table~\ref{tab:orbit_savings}.
    At level $m=m_{\max}=10$, the ghost-vertex argument of
    Remark~\ref{rk:m_max} guarantees that a single contraction per
    matroid suffices for completeness, so no orbit reduction is
    needed there.
\end{remark}

\begin{table}[ht]
    \caption{Effect of the orbit reduction at each level~$m$ for 
    the Picard number~$4$ and partial Picard number~$5$ computations.
    For each~$m$, we report the total number of non-loop 
    vertex--matroid pairs $\sum_i |\{\text{non-loop vertices 
    of }\texttt{ObjBinMat}[m][i]\}|$ and the total number of 
    orbit representatives $\sum_i k_i$ actually used by 
    Algorithm~\ref{algorithm:dynamic_programming}, together 
    with the resulting reduction factor.
    At level $m=m_{\max}=10$ for Picard number~$5$, a single 
    contraction per matroid suffices by the ghost-vertex argument 
    of Remark~\ref{rk:m_max}, so the orbit reduction does not apply.
    \label{tab:orbit_savings}}
    \centering
    \small
    \begin{tabular}{cccc}
        \toprule
        $m$ & non-loop vertices & orbit representatives 
            & reduction factor \\
        \midrule
        \multicolumn{4}{c}{\textit{Picard number $4$}} \\
        \midrule
        $7$  & $34$  & $13$ & $2.62\times$ \\
        $8$  & $47$  & $15$ & $3.13\times$ \\
        $9$  & $45$  & $14$ & $3.21\times$ \\
        $10$ & $40$  & $9$  & $4.44\times$ \\
        $11$ & $33$  & $7$  & $4.71\times$ \\
        $12$ & $24$  & $4$  & $6.00\times$ \\
        $13$ & $13$  & $2$  & $6.50\times$ \\
        $14$ & $14$  & $1$  & $14.00\times$ \\
        $15$ & $15$  & $1$  & $15.00\times$ \\
        \midrule
        total & $265$ & $66$ & $4.02\times$ \\
        \midrule
        \multicolumn{4}{c}{\textit{Picard number $5$}} \\
        \midrule
        $8$  & $114$ & $58$ & $1.97\times$\\
        $9$  & $254$ & $126$ & $2.02\times$ \\
        \midrule
        total & $368$ & $184$ & $2\times$ \\
        \bottomrule
    \end{tabular}
\end{table}

\begin{remark}
    In the implementation, $\texttt{WkPsdMfd}[m][i]$ denotes the filtered set of weak pseudomanifolds contained in $\texttt{ObjBinMat}[m][i]$ which pass the Euler-characteristic test. 
    Thus, after this filtering is applied, $\texttt{WkPsdMfd}[m][i]$ is no longer meant to contain all weak pseudomanifold subcomplexes of $\texttt{ObjBinMat}[m][i]$.

    This filtering does not remove any PL~sphere relevant to the final enumeration. 
    Indeed, if $K$ is a PL~sphere and $v$ is a non-ghost vertex of $K$, then $K/v$ is again a PL~sphere and hence has Euler characteristic
    $$
        \chi(K/v)=1+(-1)^{\dim(K/v)}.
    $$
    Therefore every inductive chain obtained by repeatedly taking vertex links of a target PL~sphere survives the Euler-characteristic filter. 
    Consequently, the filter is a pruning device only, and it does not affect the completeness of the final enumeration of mod~$2$ toric-colorable seed PL~spheres.
\end{remark}

\begin{remark}\label{rk:m_max}
    At the level $m=m_{\max}$, the post-processing step retains
    only complexes without ghost vertices. Consequently, it suffices
    to iterate over a single non-loop vertex~$v$ of~$M$ in the inductive
    step of Algorithm~\ref{algorithm:dynamic_programming}: for any
    target complex~$K$ without ghost vertices, every element of the
    ground set~$[m]$ is a vertex of~$K$, so the link $K/v$ exists and
    belongs to $\texttt{WkPsdMfd}[m-1][i']$ regardless of which
    non-loop vertex~$v$ is chosen.
    See the completeness argument in the proof of
    Proposition~\ref{prop:dynamic_programming_correctness}.
\end{remark}

\begin{remark}
    For the Picard-number-five computation reported in Theorem~\ref{thm-main}, we used $m_{\max}=10$, corresponding to $n\leq 5$.

    The remaining cases with $p=5$ and $n\geq 6$ were not enumerated
    due to computational limitations.
    Extending to $n=6$ would require setting $m_{\max}=11$.
    This introduces two additional costs.
    First, the optimization of Remark~\ref{rk:m_max} would no longer
    apply at $m=10$: since complexes with ghost vertices at that level
    must be stored as links of $m=11$ complexes, the ghost-free filter
    that justifies iterating over a single vertex~$v$ would no longer
    hold, and the computation at $m=10$ would instead require one call to
    Algorithm~\ref{algorithm:Gray_code_fixed_facets} per non-loop vertex,
    increasing the running time at that level by up to a factor of~$10$.
    Second, and more significantly, the enumeration at $m=11$ itself
    would need to be carried out: rank-$6$ cosimple binary matroids of
    corank~$5$ on $11$ elements have substantially more bases and a much
    larger kernel dimension than those at $m=10$, and the number of
    weak pseudomanifolds at that level is expected to be orders of
    magnitude larger than the $78{,}933{,}520$ entries in the $n=5$ row
    of Table~\ref{table:numbers_pic5}.
\end{remark}

\subsection{Postprocessing and enumeration of the seeds}\label{sec:postprocessing}

After running Algorithm~\ref{algorithm:dynamic_programming}, we have, for each cosimple binary matroid of corank~$p$, all weak pseudomanifold subcomplexes satisfying the mod~$2$ nonsingularity condition.
Some of these subcomplexes may have ghost vertices.
After deleting ghost vertices, their Picard numbers are at most $p$.

All that remains to do is to find which ones are seeds that are PL~spheres, as well as finding one representative per isomorphism class.
\medskip

First, we need a criterion for a weak pseudomanifold to be a PL~sphere.
When the Picard number is small enough, there is a nice characterization of PL~manifolds that are PL~spheres.
\begin{theorem}[\cite{Bagchi-Datta2005}]\label{theorem:combisp}
	Let $K$ be a PL~manifold such that $m-n\leq 7$.
	If $K$ is a $\Z_2$-homology sphere, then it is a PL~sphere.
\end{theorem}

Recall that a pure simplicial complex is a closed PL~manifold if and only if every vertex link is a PL~sphere.
Once $K$ has been verified to be a $\Z_2$-homology sphere, it remains to certify that $K$ is a closed PL~manifold.
We do this by recursively certifying that every vertex link is a PL~sphere, using the previously constructed seed database.
After this PL~manifold property is certified, Theorem~\ref{theorem:combisp} implies that $K$ is a PL~sphere.
\begin{algorithm}
    \caption{Certifying that a $\Z_2$-homology sphere of Picard number at most $7$ is a PL~sphere by checking vertex links against the seed database.\label{algo:test_PLS}}
    \KwData{A $\Z_2$-homology sphere $K$ on $m$ non-ghost vertices, of dimension $n-1$; the database $\texttt{TCSeeds}$ of seed PL~spheres on strictly fewer than $m$ vertices.}
    \KwResult{$\texttt{true}$ if $K$ is certified a PL~sphere (i.e. every vertex link of $K$ has a seed isomorphic to some element of $\texttt{TCSeeds}$), $\texttt{false}$ otherwise.}
    \For{$v$ a non-ghost vertex of $K$}{
        $K/v \gets \texttt{Link}(K,v)$
        \tcp*{the link of $v$ in $K$, on $m-1$ vertices, of dimension $n-2$}
        $(L,\,m',\,n') \gets \texttt{Seed}(K/v)$
        \tcp*{$L$: seed of $K/v$; $m'\le m-1$ non-ghost vertices; $n' \leq n-1$}
        \If{$\texttt{TCSeeds}[(m',n')] = \emptyset$ or $L$ is not isomorphic to any element of $\texttt{TCSeeds}[(m',n')]$}{
            \Return\texttt{false}\;
        }
    }
    \Return $\texttt{true}$\;
\end{algorithm}

The procedure $\texttt{Seed}$ used in Algorithm~\ref{algo:test_PLS} has the following property: whenever 
$$
    (L,m',n')=\texttt{Seed}(K/v),
$$
the complex $L$ is a seed on $m'$ non-ghost vertices and of dimension $n'-1$, and, after deleting ghost vertices, $K/v$ is obtained from $L$ by a sequence of wedge operations. 

\begin{proposition} \label{prop:PLS-test}
    Let $K$ be a mod~$2$ toric-colorable $\Z_2$-homology sphere on $m$ non-ghost vertices, of dimension $n-1$, with $m-n\le 7$.
    Suppose that, for every pair $(m',n')$ with $m'<m$ and $m'-n'\le 7$, the database $\texttt{TCSeeds}[(m',n')]$ contains exactly one representative of each isomorphism class of mod~$2$ toric-colorable seed PL~spheres on $m'$ non-ghost vertices and of dimension~$n'-1$, and contains no other complexes.
    Then Algorithm~\ref{algo:test_PLS} returns true on input $K$ if and only if $K$ is a PL~sphere.
\end{proposition}
\begin{proof}
    Suppose first that Algorithm~\ref{algo:test_PLS} returns true. 
    Then, for every non-ghost vertex~$v$ of $K$, if $(L,m',n')=\texttt{Seed}(K/v)$, then $L$ is isomorphic to an element of $\texttt{TCSeeds}[(m',n')]$.
    By the assumption on $\texttt{TCSeeds}$, the complex $L$ is a seed PL~sphere.

    By the assumed property of $\texttt{Seed}$, after deleting ghost vertices, $K/v$ is obtained from $L$ by a sequence of wedge operations.
    Since wedge operations preserve the PL~sphere property, and ghost vertices do not affect the geometric realization, $K/v$ is a PL~sphere. 
    Hence every vertex link of $K$ is a PL~sphere. 
    Therefore $K$ is a closed PL~manifold.

    Since $K$ is a $\Z_2$-homology sphere and $m-n\le 7$, Theorem~\ref{theorem:combisp} implies that $K$ is a PL~sphere.

    Conversely, suppose that $K$ is a PL~sphere. 
    Let $v$ be a non-ghost vertex of~$K$. 
    Then $K/v$ is a PL~sphere. 
    Moreover, since $K$ is mod~$2$ toric-colorable, the link $K/v$ is also mod~$2$ toric-colorable by projection of a mod~$2$ characteristic map, equivalently by contraction of the associated binary matroid.

    Let $(L,m',n')=\texttt{Seed}(K/v)$.
    Then $L$ is a PL~sphere. 
    Indeed, each inverse wedge step can be realized as taking the link of one of the two wedged vertices, and links of vertices in PL~spheres are PL~spheres.
    The same argument shows that $L$ remains mod~$2$ toric-colorable, because mod~$2$ toric-colorability is preserved under taking links by projection.

    The Picard number does not increase when passing to vertex links, and inverse wedge operations preserve the Picard number. 
    Hence 
    $$
        m'-n'\le m-n\le 7.
    $$
    Also $m'<m$. 
    By the assumption on~$\texttt{TCSeeds}$, the seed~$L$ is isomorphic to an element of $\texttt{TCSeeds}[(m',n')]$. 
    Thus the test in Algorithm~\ref{algo:test_PLS} succeeds for the vertex~$v$. 
    Since this holds for every non-ghost vertex $v$, Algorithm~\ref{algo:test_PLS} returns true.
\end{proof}

For a fixed $p$, we will store in $\texttt{TCSeeds}[(m,n)]$ one representative of each isomorphism class of mod~$2$ toric-colorable seed PL~spheres with $m$ non-ghost vertices and dimension $n$, where $m-n\le p$.

We start by populating $\texttt{TCSeeds}$ with all boundaries of cross-polytopes having Picard number at most $p$ since they are the only seeds that are included in no cosimple binary matroids by Corollary~\ref{cor:non_IDCM_colorable_seeds}.

We finally proceed as follows.

\noindent\textbf{Pre-filters.}
We first populate a data structure $\texttt{SeedZ2HomSphere}$ such that\\ $\texttt{SeedZ2HomSphere}[(m,n)]$ contains all mod~$2$ toric-colorable weak pseudomanifolds which are seeds and $\Z_2$-homology spheres, on $m$ non-ghost vertices and of dimension~$n$.
\begin{enumerate}
    \item For every cosimple binary matroid $\texttt{ObjBinMat}[m][i]$, compute its automorphism group $G$ and keep only the lexicographically minimal weak pseudomanifolds with respect to $G$ in $\texttt{WkPsdMfd}[m][i]$.
    \item For every $K$ in $\texttt{WkPsdMfd}[m][i]$, if $K$ is a $\Z_2$-homology sphere and a seed, add it to $\texttt{SeedZ2HomSphere}[(m',n')]$ for appropriate $(m',n')$.
\end{enumerate}
\noindent\textbf{Isomorphism and PL~sphere property.}
We then populate the final data structure $\texttt{TCSeeds}$.
\begin{enumerate}
    \item[(3)] For increasing values of $m$, and for decreasing $n=m-1,\ldots,m-p$, consider every $K\in\texttt{SeedZ2HomSphere}[(m,n)]$.
    \begin{enumerate}
        \item[(3.1)] If Algorithm~\ref{algo:test_PLS} with input $K$ and $\texttt{TCSeeds}$ returns $\texttt{false}$, discard $K$.
        \item[(3.2)] If $K$ is not isomorphic to any element of $\texttt{TCSeeds}[(m,n)]$, insert it into $\texttt{TCSeeds}[(m,n)]$.
    \end{enumerate}
\end{enumerate}
The order in which we go through $\texttt{SeedZ2HomSphere}$ certifies that for a fixed $m$, the database $\texttt{TCSeeds}$ contains exactly the seeds of all PL~spheres on $m'<m$ vertices, which by Proposition~\ref{prop:PLS-test} certifies that $K$ is a PL~sphere.

At the end of this process, the table $\texttt{TCSeeds}$ contains one representative of each isomorphism class of mod~$2$ toric-colorable seeds of Picard number at most $p$ detected by the cosimple matroid search, together with the cross-polytope boundaries added at the initial step.

\subsection{Results}
Using this framework, we reduced the weak pseudomanifold enumeration stage for the complete Picard number four computation from more than $10$ days in the GPU-based approach of~\cite{Choi-Jang-Vallee2024} to $10$ minutes on a single CPU.
The automorphism reduction step together with bucketing reduced the postprocessing part from half a day to less than an hour on a single CPU.
Note that we reproduced the exact same mod~$2$ toric-colorable seeds with Picard number four, which gives an independent verification of the data of~\cite{Choi-Jang-Vallee2024}.

For Picard number five, we parallelized the Gray code enumeration by fixing several initial bits.
We ran the algorithm using $64$ threads\footnote{Nic5 at University of Liège (SEGI), two AMD EPYC 7542 (Rome) 32-core CPUs at 2.9 GHz} and the enumeration finished in about one day.
The postprocessing step finished in around 30 hours.

Table~\ref{tab:kernel_dim} illustrates why the dynamic 
programming approach is essential at level $m=10$.
A direct application of Algorithm~\ref{algorithm:Gray_code} 
to each cosimple binary matroid $M$ on $10$ elements would 
require enumerating its full mod~$2$ kernel of $A(M)$, at a 
cost of $\cO(Fn\cdot 2^d)$ per matroid, with $d$ ranging 
from $16$ to $60$ across the $46$ matroids at this level.
Such an enumeration is computationally infeasible for the 
matroids with large~$d$.

The dynamic programming approach instead decomposes each such 
enumeration into many calls to 
Algorithm~\ref{algorithm:Gray_code_fixed_facets}, each seeded 
by a cone constraint from a previously computed link.
Unit propagation in Phase~$1$ then fixes or forbids additional 
facets, reducing the effective search space per call from 
$2^d$ to $2^{|\cI^\circ|}$.
As Table~\ref{tab:kernel_dim} shows, the reduction 
$d-|\cI^\circ|$ has mean $30.15$ and median $30$, 
corresponding to a typical reduction in search space by a 
factor of $2^{30}\approx 10^9$ per call compared with a 
naive enumeration of the full kernel.
While individual calls reach up to $|\cI^\circ|=41$, the 
median of $|\cI^\circ|=16$ confirms that the vast majority 
of calls are inexpensive, and it is this concentration that 
makes the overall computation feasible within a single day 
on $64$ threads.

\begin{table}[ht]
    \caption{Statistics on the mod~$2$ kernel dimension~$d$ of 
    $A(M)$ across the $46$ cosimple binary matroids of corank~$5$ 
    on $m=10$ elements, on the number $|\cI^\circ|$ of free 
    generators after unit propagation, and on their difference 
    $d-|\cI^\circ|$, measured across all calls to 
    \texttt{GrayCodeExtend} using one contraction per matroid.
    The difference $d-|\cI^\circ|$ quantifies the reduction in 
    search space achieved by the dynamic programming cone 
    constraint and unit propagation combined, replacing a 
    $2^d$-step enumeration by one of $2^{|\cI^\circ|}$ steps 
    per call.\label{tab:kernel_dim}}
    \centering
    \small
    \begin{tabular}{lcccc}
        \toprule
         & min & mean & median & max \\
        \midrule
        Kernel dimension $d$ of $A(M)$
            & 16 & 33.59 & 32 & 60 \\
        Free generators $|\cI^\circ|$ per call
            & 2  & 13.83 & 16 & 41 \\
        Reduction $d - |\cI^\circ|$
            & 11& 30.15 & 30 & 36 \\
        \bottomrule
    \end{tabular}
\end{table}

Tables~\ref{table:numbers_pic4} and~\ref{table:numbers_pic5} describe the number of weak pseudomanifolds obtained at each step of the postprocessing.
More precisely, the quantity displayed represents before step \begin{enumerate}
    \item[(1):] the number of mod~$2$ toric-colorable weak pseudomanifolds which satisfy the Euler test obtained via the dynamic programming algorithm,
    \item[(2):] the number of mod~$2$ toric-colorable weak pseudomanifolds which satisfy the Euler test up to automorphism of each binary matroid it is included in,
    \item[(3):] the number of mod~$2$ toric-colorable and $\Z_2$-homology spheres which are wedge-minimal,
    \item[final:] the number of mod~$2$ toric-colorable seed PL~spheres up to isomorphism.
\end{enumerate}

\begin{table}[ht]
    \caption{Number of objects obtained before each step and at the end of the postprocessing steps of the enumeration for Picard number $4$. Note that the Euler test was used only up to $n=7$.}
    \label{table:numbers_pic4}
    \centering
    \small
    \setlength{\tabcolsep}{5pt}
    \begin{tabular}{c c c c c }
    \toprule
    $n$
    & before step~(1) & before step~(2) & before step (3) & final result\\
    \midrule
    $2$  &  $131$   &   $5$  &  $4$  & $1$        \\
    $3$  &   $1{,}164$  &   $43$  &   $36$  & $4$          \\
    $4$  &   $10{,}127$  &   $316$  &   $181$  & $21$  \\
    $5$  &   $31{,}731$  &   $1{,}330$  &   $744$  & $142$  \\
    $6$  & $78{,}802$    &   $2{,}997$  &   $1{,}701$  & $733$ \\
    $7$  &  $147{,}819$   &  $3{,}531$   &  $2{,}054$   & $1{,}190$  \\
    $8$  &   $237{,}401$  &  $2{,}518$   &   $1{,}243$  & $776$     \\
    $9$  &    $128{,}173$ &    $746$ &   $366$  & $243$  \\
    $10$ &   $141{,}179$  & $143$    &   $54$  & $39$      \\
    $11$ &  $154{,}440$   &   $24$  &  $5$   & $4$      \\
    \bottomrule
    \end{tabular}
\end{table}

\begin{table}[ht]
    \caption{Number of objects obtained before each step and at the end of the postprocessing steps of the enumeration for Picard number $5$.}
    \label{table:numbers_pic5}
    \centering
    \small
    \setlength{\tabcolsep}{5pt}
    \begin{tabular}{c c c c c }
    \toprule
    $n$
    & before step~(1) & before step~(2) & before step (3) & final result\\
    \midrule
    $2$  &  $657$   &   $5$  &  $3$ & $1$        \\
    $3$  &   $19{,}360$  &   $488$  &  $441$  & $13$          \\
    $4$  &   $1{,}305{,}761$  &   $46{,}721$  &  $43{,}010$   & $1{,}170$  \\
    $5$  &   $78{,}933{,}520$  &   $3{,}062{,}851$  &    $2{,}688{,}620$  & $198{,}846$  \\
    \bottomrule
    \end{tabular}
\end{table}

\section{Toric-colorability for the case $(n,m)=(5,10)$.}\label{sec:toric-colorable}

For each mod~$2$ toric-colorable seed PL~sphere $K$ with $m=10$ vertices and $\dim K=4$, we checked whether $K$ admits an integral characteristic map whose columns are $0/1$-vectors. 
More precisely, we searched for a map 
$$
    \lambda \colon V(K) \longrightarrow \{0,1\}^{5}\setminus\{0\}
$$
such that for every facet $\sigma=\{v_1,\ldots,v_5\}$ of $K$,
$$
    \left|\det\bigl(\lambda(v_1),\ldots,\lambda(v_5)\bigr)\right|=1.
$$
Such a map is an integral characteristic map, and hence proves that $K$ is toric-colorable.

The search was performed by an exact finite backtracking algorithm.  
First, we enumerated all unordered $5$-subsets
$$
    B\subset \{0,1\}^{5}\setminus\{0\}
$$
with $|\det B|=1$.
There are $80{,}856$ such admissible subsets, obtained by checking 
all $\tbinom{31}{5}=169{,}911$ unordered $5$-subsets of 
$\{0,1\}^5\setminus\{0\}$ for the determinant condition.
During the backtracking search, vertices were assigned non-zero $0/1$ columns one at a time.  
Whenever all vertices of a facet had been assigned, the corresponding set of five columns was required to be one of the admissible subsets.  
More generally, a partial assignment was discarded as soon as some partially assigned facet could not be extended to an admissible subset.  
Thus the algorithm exhaustively checks the existence of a $0/1$ integral characteristic map.

Every complex among the $198{,}846$ mod~$2$ toric-colorable seed PL~spheres in the case $(n,m)=(5,10)$ passed the test.
Therefore all mod~$2$ toric-colorable seed PL~spheres in our $(n,m)=(5,10)$ data set are toric-colorable.  
In fact, each of them admits a $0/1$ integral characteristic map.
This yields the counts stated in Theorem~\ref{thm-main}.

\begin{remark}
    It is an open problem whether every mod~$2$ toric-colorable 
    PL~sphere is toric-colorable. Our computation provides 
    computational evidence toward a positive answer in the case 
    $(n,m)=(5,10)$: all $198{,}846$ mod~$2$ toric-colorable seeds 
    of dimension four and Picard number five admit a $\{0,1\}$-integral 
    characteristic map. Whether this persists in higher Picard numbers 
    remains open.\\
    A stronger open problem, known as the \emph{lifting property} and 
    introduced by Zhi L\"u at the toric topology conference held in Osaka in 2011 (as documented in~\cite{Choi-Park2016Wedge}), asks whether every 
    mod~$2$ characteristic map over a mod~$2$ toric-colorable 
    PL~sphere is the mod~$2$ reduction of some $\Z$-characteristic map. A positive answer to the lifting problem would imply toric-colorability for every mod~$2$ toric-colorable PL~sphere.
\end{remark}

\providecommand{\bysame}{\leavevmode\hbox to3em{\hrulefill}\thinspace}
\providecommand{\MR}{\relax\ifhmode\unskip\space\fi MR }
\providecommand{\MRhref}[2]{%
  \href{http://www.ams.org/mathscinet-getitem?mr=#1}{#2}
}
\providecommand{\href}[2]{#2}

\end{document}